\documentclass[11pt]{amsart}
\usepackage{amssymb,amscd,latexsym,epsfig,hyperref,xcolor}
\usepackage[curve,color]{xypic}
\usepackage{scalerel}
\usepackage{stackengine,wasysym}
\usepackage{multicol}

\DeclareFontFamily{OT1}{rsfs}{}
\DeclareFontShape{OT1}{rsfs}{n}{it}{<-> rsfs10}{}
\DeclareMathAlphabet{\curly}{OT1}{rsfs}{n}{it}

\makeatletter
\newcommand{\eqnum}{\refstepcounter{equation}\textup{\tagform@{\theequation}}}
\makeatother

\newcommand{\be}{\begin{equation}}
\newcommand{\ee}{\end{equation}}
\newcommand{\benn}{\begin{equation*}}
\newcommand{\eenn}{\end{equation*}}
\newcommand{\bea}{\begin{eqnarray}}
\newcommand{\eea}{\end{eqnarray}}
\newcommand{\ba}{\begin{align}}
\newcommand{\ea}{\end{align}}
\newcommand{\bi}{\begin{itemize}}
\newcommand{\ei}{\end{itemize}}

\newcommand{\mc}{\mathcal}
\newcommand{\mbb}{\mathbb}

\newcommand{\D}{D}
\newcommand{\cO}{\mathcal O}

\newcommand{\mcL}{\mathcal L}

\newcommand{\irrdivcoeff}{\beta}
\newcommand{\shfdiv}[1]{\tilde{#1}}

\newcommand\reallywidetilde[1]{\ThisStyle{%
  \setbox0=\hbox{$\SavedStyle#1$}%
  \stackengine{-.1\LMpt}{$\SavedStyle#1$}{%
    \stretchto{\scaleto{\SavedStyle\mkern.2mu\AC}{.5150\wd0}}{.6\ht0}%
  }{O}{c}{F}{T}{S}%
}}

\newcommand{\ind}{\chi}

\newcommand{\ceil}[1]{\mathrm{ceil}\left(#1\right)}

\makeatletter \@addtoreset{equation}{section} \makeatother
\renewcommand{\theequation}{\thesection.\arabic{equation}}
\newtheorem{defn}[equation]{Definition}
\newtheorem{thm}[equation]{Theorem}
\newtheorem{thm*}{Theorem}
\newtheorem{lem}[equation]{Lemma}

\newtheorem{prop*}{Proposition}

\newenvironment{rmk}{\noindent\textbf{Remark}.}{\medskip}

\newtheorem{crl}[equation]{Corollary}

\title[Topological Formulae for Line Bundle Cohomology on Surfaces]{Topological Formulae for the Zeroth Cohomology of Line Bundles on del Pezzo and Hirzebruch Surfaces}
\author[Callum R. Brodie, Andrei Constantin, Rehan Deen and Andre Lukas]{Callum R. Brodie, Andrei Constantin, \\[4pt] Rehan Deen and Andre Lukas}

\begin{document}
\maketitle
\begin{abstract} We show that the zeroth cohomology of effective line bundles on del Pezzo and Hirzebruch surfaces can always be computed in terms of a topological index. 

%
\end{abstract}

\renewcommand\contentsname{\vspace{-9mm}}
\tableofcontents 




\section{Introduction and summary of results} \label{intro}
It is usually hard to compute the cohomology of a holomorphic line bundle $L$ over a complex manifold $X$. In the  situation when $L$ is ample and $X$ a smooth complex projective variety, Kodaira's vanishing theorem ensures that all cohomology groups $H^i(K_X\otimes L)$ with indices $i>0$ are automatically zero, where $K_X$ is the canonical line bundle of $X$. In this case, the dimension of the zeroth cohomology of $K_X\otimes L$ can be computed as a topological index. 

\textcolor{black}{In this paper we study certain classes of non-singular complex projective surfaces for which we show that the region in the Picard lattice where the dimension of the zeroth cohomology is given by a topological index can be extended to the entire effective cone.} 
In particular, for surfaces such as del Pezzos and toric surfaces whose fans have convex support, the statement holds true thoughout the entire Picard lattice, leading to closed form expressions for the dimension of the zeroth cohomology and hence for all higher cohomologies. 

\textcolor{black}{The validity of these results is derived from the existence of a map taking effective divisor classes to nef divisor classes while preserving the dimension of the zeroth cohomology.}
\begin{thm}\label{thm:h0invariance}
Let $\mathcal{L}\rightarrow S$ be an effective line bundle on a non-singular projective surface~$S$, and $D$ an effective divisor such that $\cO(D)=\mcL$. 
Let $\{D_i\}$ be the finite collection of irreducible negative divisors 
with $D \cdot D_i < 0$. 
 \textcolor{black}{Then the map $D\mapsto \tilde D$, where $\tilde D$ is defined by}
\begin{equation*}
\tilde D = D \,-\!\!\!\! \sum_{i:\,D\cdot D_i<0} \ceil{\frac{ D_i \cdot D } { D_i^2 }} D_i ~,
\end{equation*}
preserves the dimension of the zeroth cohomology, so that:
\begin{equation*}
h^0(S,\cO(\tilde D)) = h^0(S, \mcL) \,.
\end{equation*}
\end{thm}

\textcolor{black}{The map $D\mapsto\tilde D$ can be discussed in terms of the Zariski decomposition of effective divisors and its properties. However, we prefer a more elementary treatment, which is sufficient for the present purpose.} The idea of the proof relies on the basic observation that a fixed component $F$ of the complete linear system $|D|$, i.e.~a divisor contained in every element of $|D|$, does not contribute to the dimension count and hence ${\rm dim}|D-F| = {\rm dim}|D|$ or equivalently, 
\begin{equation*}
h^0(S,\cO(D-F)) = h^0(S, \cO(D)) \,.
\end{equation*}

While in general it may be hard to extract the entire fixed part of a linear system, some of the fixed components are immediate: if $[D]\cdot [D_i]<0$, then $D_i$ must be contained in every element of $|D|$, otherwise there exists $D'\in |D|$ such that $D'\cdot D_i\geq 0$, which is a contradiction. 
Theorem~\ref{thm:h0invariance} gives a lower bound on the number of times that a negative divisor is contained in the base locus of $|D|$. 

\textcolor{black}{We call the map $D\mapsto\tilde D$ an isoparametric transform, a terminology reminiscent of the notion of parameters in a family of curves.}
This map eliminates some of the fixed components of $|D|$. In general, it changes the class $[\tilde D]\neq[D]$ and hence the corresponding line bundle $\cO(\tilde D)\neq \cO(D)$, while preserving the dimension of the complete linear system, ${\rm dim} |\tilde D| = {\rm dim} |D|$. 

\begin{thm}\label{thm:nef}
After a finite number of iterations, the isoparametric transform outputs in the nef cone. Subsequently, it remains constant.  
\end{thm}

\medskip\noindent\textbf{Topological formulae for the zeroth cohomology  (general case).}
The nef cone always contains a region where the zeroth cohomology equals the index, though the exact expanse of this region depends on the particulars of the surface in question. In general, if $S$ is a non-singular complex projective surface, Kodaira's vanishing theorem ensures that this region includes
\be\label{indexRegion}
 ({\rm Amp}(S)+K_S)\,\cap\,{\rm Nef}(S) \,\cap\, {\rm Pic}(S)~,
\ee
containing the lattice points in a shifted cone of full-dimension.

If the limit $\underline{\tilde D}$ of the iterated isoparametric transform belongs to this region, the following index-like formula for the zeroth cohomology holds 
\begin{equation}\label{IndexFormula}
h^0(S, \cO(D)) = \chi (S, \cO(\underline{\tilde D}))~, 
\end{equation}
relating an analytical quantity, the zeroth cohomology of $\cO(D)$ to a purely topological quantity, the Euler characteristic of $\cO(\underline{\tilde D})$.

\medskip\noindent\textbf{Topological formulae on other surfaces.} For particular surfaces, stronger vanishing theorems can enlarge the region~\eqref{indexRegion} to the entire nef cone intersected with the Picard lattice. 

\begin{thm}[Kawamata-Viehweg vanishing theorem for surfaces]
\label{thm:kawamataviehweg}
Let $S$ be a smooth complex projective surface, and let $D$ be a nef and big divisor on $S$. Then
\benn
h^i\left(S,\mc{O}_S(K_S+D)\right) = 0 \quad \mathrm{for}~i>0 \,.
\eenn
\end{thm}

\begin{crl}\label{crl:KV}
If the limit $\underline{\tilde D}$ of the isoparametric transform applied to an effective divisor $D$ is such that $\underline{\tilde D} - K_S$ is nef and big, the topological formula~\eqref{IndexFormula} follows by Kawamata-Viehweg. In particular, this is always the case for surfaces with nef and big anti-canonical divisor class, i.e.~$-K_S\cdot C\geq 0$ for all curves $C$ on $S$ and $(-K_S)^2>0$, which includes del Pezzo surfaces.
\end{crl}

Del Pezzo surfaces are discussed in detail as an example in Section~\ref{sec:Examples}. In particular, we show that the index formula~\eqref{IndexFormula} implies the following closed form expression for the zeroth cohomology of an arbitrary effective line bundle $\cO(D)$:
\benn
h^0\left(\mc{O}(D)\right) = \ind\bigg(D +\sum_{i}(D\cdot D_i)D_i \bigg) ~,
\eenn
where the index $i$ runs over the set of all irreducible negative divisors such that $D \cdot D_i < 0$. 

\vspace{12pt}
Separately, a large class of toric surfaces can be treated using the following theorem. 

\begin{thm}[Demazure vanishing theorem for surfaces]
\label{thm:demazure}
Let $S$ be a toric surface whose fan has convex support, and let $D\in {\rm Nef}(S)$. Then
\benn
h^i\left(S,\mc{O}_S(D)\right) = 0 \quad \mathrm{for}~i>0 \,.
\eenn
\end{thm}

\begin{crl}
Since for a compact toric surface the fan has convex support, the topological formula~\eqref{IndexFormula} applies in this case to any effective divisor $D$. 
\end{crl}

As an illustration, in Section~\ref{Hirzebruch} we discuss Hirzebruch surfaces of degree~$n$ and show that the zeroth cohomology of an effective line bundle $\mc{O}(D)$ is given by
\benn
h^0\left(\mc{O}(D)\right) = \ind\bigg(D - \theta(-D \cdot D_{-n})\,\ceil{\frac{D\cdot D_{-n}}{-n}}D_{-n} \bigg) \,,
\eenn
where $D_{-n}$ is the unique irreducible divisor with negative self-intersection $D_{-n}^2=-n$ (not present for $n=0$) and 
\begin{equation*}
\theta(x) = \begin{cases}1,~~x\geq 0\\ 0,~~ x< 0\end{cases}  
\end{equation*}
is the Heaviside step-function.

\medskip\noindent\textbf{Topological formulae on K3 surfaces.} If $K_S$ is trivial, the region~\eqref{indexRegion} becomes
\benn
{\rm Amp}(S) \,\cap\, {\rm Pic}(S)~.
\eenn
This implies that the topological formula~\eqref{IndexFormula} holds for all effective line bundles except those lying on the boundary of the nef cone or preimages thereof under the isoparametric transform.

\medskip\noindent\textbf{Acknowledgements.} We would like to thank Damian R\"ossler for inspiring discussions.  The present article is complemented by two other papers by the same authors \cite{BCDL1, BCDL2}, the first dealing with physical applications and the second showing that machine learning algorithms can effectively be implemented in order to find analytic formulae for line bundle cohomology. 
Analytic formulae similar to the ones presented here have been found in other contexts, such as complete intersection Calabi-Yau threefolds \cite{Co}, \cite{BCL}, \cite{CL}, \cite{LS}, and (hypersurfaces in) toric varieties \cite{KS}.

\section{The isoparametric transform}\label{sec:proofs}

Fix a non-singular complex projective algebraic surface $S$ and let $\mathcal{L}\rightarrow S$ be an effective line bundle. Knowing the dimension of the zeroth cohomology for all effective line bundles is enough in order to infer all other cohomology dimensions. The $h^2(S,\mcL)$ cohomology is obtained via Serre duality and $h^1(S,\mcL)$ follows by the index theorem. 

The isoparametric transform, discussed below, maps effective line bundles to nef line bundles. This proves useful in the cases when the cohomology of nef line bundles can be accessed through vanishing theorems. 

\begin{defn}\label{def:itransform} 
Let $\mathcal{L}\rightarrow S$ be an effective line bundle and $D$ an effective divisor such that $\cO(D)=\mcL$. Denote by $\{D_i\}$ the finite collection of irreducible negative divisors with $D \cdot D_i < 0$. 
Then $D$ can be written uniquely as 
\begin{equation*}
D = \tilde{D} + F~,
\end{equation*}
where $\tilde D$ is effective and $F$ is a fixed component of the complete linear system of $D$ defined by
\begin{equation*}
F \,=\!\!\!\!  \sum_{i:\,D\cdot D_i<0} \ceil{\frac{ D_i \cdot D } { D_i^2 }} D_i ~.
\end{equation*}
$\tilde D$ is the isoprametric transform of $D$ and $\tilde \mcL = \cO(\tilde D)$ the isoparametric transform of $\mcL$.
\end{defn}

\begin{rmk}
The isoparametric transform is well defined as a map between line bundles, since it depends only on the class $[D]$.
\end{rmk}

Theorem \ref{thm:h0invariance} asserts the invariance of the zeroth cohomology  dimension under the transform.

\vspace{4pt}
\medskip\noindent\textbf{Proof of Theorem \ref{thm:h0invariance}.} 
For any effective divisor $D$ such that $\cO(D)=\mcL$, as $D_i$ is irreducible and negative, the condition $D\cdot D_i<0$ guarantees that $D_i$ is contained in $D$, hence it is a fixed component of the complete linear system~$|D|$. To obtain a bound on its coefficient, write $D$ as
\benn
D = \sum_a \irrdivcoeff_a D_a \,,
\eenn
where $D_a$ are pairwise distinct irreducible divisors and $\irrdivcoeff_a \geq 0$. Then
\benn
0 > D\cdot D_i  = \sum_a \irrdivcoeff_a D_a\cdot D_i = \irrdivcoeff_i D_i^2 + (\textrm{non-negative terms}) \geq \irrdivcoeff_i D_i^2 \,,
\eenn
which holds since $D_a\cdot D_i$ cannot be negative when $D_a$ and $D_i$ are distinct, by irreducibility. 
This implies the bound
\be\label{eq:bound}
\irrdivcoeff_i \geq \ceil{\frac{ D\cdot D_i} { D_i^2 }} \,.
\ee
This holds for any irreducible divisor $D_i$ that intersects $D$ negatively. 
Since the bound \eqref{eq:bound} depends on intersection products only, it must apply to any element of $|D|$, hence 
\begin{equation*}
F \,=\!\!\!\!  \sum_{i:\,D\cdot D_i<0} \ceil{\frac{ D_i \cdot D } { D_i^2 }} D_i ~
\end{equation*}
is a fixed component of the complete linear system $|D|$ and does not contribute to the dimension count. Consequently, 
\benn
{\rm dim} |D|={\rm dim} |D-F|~, 
\eenn
or, equivalently, 
\benn
h^0(S, \cO(D)) = h^0(S,\cO(D-F))~.
\eenn

\subsection{Iteration and abutment}
Starting with an effective divisor $D$, the isoparametric transform stops after a finite number of iterations, the resulting class being nef. This is clear since any iteration of the isoparametric transform eliminates a fixed component of $|D|$, of which there are finitely many. The process stops precisely when $\tilde D\cdot D_{\rm eff}\geq 0$ for any effective divisor $D_\mathrm{eff}$, which means $\tilde D$ is nef. This proves Theorem \ref{thm:nef}. 

The limit of the iterated isoparametric transform applied to an effective divisor $D$ will be denoted by $\underline{\tilde D}$. It is clear that the class $[\underline{\tilde D}]$ depends only on the class~$[D]$, hence the iterated isoparametric transform maps effective line bundles to nef line bundles, while preserving the zeroth cohomology dimension. We will use the notation $\underline{\tilde \mcL}$ for the nef line bundle associated with the effective line bundle $\mcL$. 

The iterated isoparametric transform can be practically implemented in the following two cases:
\begin{itemize}
\item[1.] The effective line bundle $\mcL$ comes with a section and the curve decomposition of the corresponding divisor is known.
\item[2.] The surface $S$ carries only finitely many negative curves. In particular, this holds true if the effective cone of $S$ is finitely generated. 
\end{itemize}

In general, several iterations of the isoparametric transform are required to achieve nefness. The cases in which the isoparametric transform abuts after one iteration are interesting as they lead to closed form expressions for cohomology. This is trivially the case of surfaces that contain a single negative curve, such as Hirzebruch surfaces, discussed in Section~\ref{Hirzebruch}. It is also the case for surfaces that contain only negative curves with self-intersection equal to $-1$, such as del Pezzo surfaces. To show this, we need the following results.

\begin{lem}
Let $D$ be an effective divisor on a smooth complex projective surface, and write $\{D_i\}$ for the finite collection of irreducible divisors $D_i$ with $D\cdot D_i < 0$. Then $D_i \cdot \D_{j \neq i} < \mathrm{max}(-D_i^2,-D_j^2)$.
\end{lem}
\begin{proof}
Let $D_i$ and $D_j$ be as above and write the curve decomposition of $D$
\benn
D = \irrdivcoeff_i D_i + \irrdivcoeff_j D_j + D_\mathrm{eff.}~,
\eenn
where $D_\mathrm{eff.}$ is an effective divisor and the integers $\irrdivcoeff_i, \irrdivcoeff_j \geq 0$. Then
\benn
\begin{aligned}
0 > D \cdot D_i &= \irrdivcoeff_i(D_i^2) + \irrdivcoeff_j (D_i \cdot D_j) + (\textrm{non-negative terms}) \,, \\
0 > D \cdot D_j &= \irrdivcoeff_i (D_i \cdot D_j) + \irrdivcoeff_j(D_j^2) + (\textrm{non-negative terms}) \,.
\end{aligned}
\eenn
Hence
\benn
\begin{aligned}
0 &> \irrdivcoeff_i(D_i^2) + \irrdivcoeff_j (D_i \cdot D_j) \quad \Rightarrow \quad \irrdivcoeff_i > \left(\frac{D_i \cdot D_j}{-D_i^2}\right)\irrdivcoeff_j \,, \\
0 &> \irrdivcoeff_i (D_i \cdot D_j) + \irrdivcoeff_j(D_j^2) \quad \Rightarrow \quad \irrdivcoeff_j > \left(\frac{D_i \cdot D_j}{-D_j^2}\right)\irrdivcoeff_i \,.
\end{aligned}
\eenn
But these are inconsistent unless at least one of the bracketed factors is less than one, which proves the theorem.
\end{proof}

\begin{crl}\label{DiDj}
Under the same conditions as the lemma, if there are no irreducible divisors with self-intersection $<-1$, then $D_i \cdot D_{j} = -\delta_{ij}$.
\label{crl:zeroints}
\end{crl}
\begin{proof}
This is immediate.
\end{proof}

\begin{thm}\label{Thm:oneStep}
On smooth projective surfaces containing no negative curves with self-intersection $<-1$ the isoparametric transform applied to an effective divisor abuts after at most one application. 
\end{thm}

\begin{proof}

As $D$ is effective, write its curve decomposition as
\benn
D = \sum_i \irrdivcoeff_i D_i + \sum_\alpha \irrdivcoeff_\alpha D_\alpha + \sum_\mu \irrdivcoeff_\mu D_\mu \,,
\eenn
where $i,\alpha,\mu$ index, respectively, the irreducible divisors such that $D \cdot D_i < 0$, those such that $D_\alpha^2 < 0$ but $D \cdot D_\alpha \geq 0$, and the remainder. As usual, denote by $\tilde D$ the isoparametric transform of $D$,
\benn
\tilde D = D \, + \sum_{i} (D_i \cdot D)\,  D_i ~.
\eenn
To show that $\tilde D$ is nef, first note that  $\shfdiv{D} \cdot D_\mu \geq 0$. Furthermore, note that
\benn
\shfdiv{D} \cdot D_j = D\cdot D_i + \sum_{i} (D_i \cdot D)\,  (D_i\cdot D_j) = 0
\eenn
by Corollary~\ref{DiDj}. 
It remains only to show that $\shfdiv{D} \cdot D_\alpha \geq 0$. Collect the terms proportional to $D_i$ in $\shfdiv{D}$
\benn
\begin{aligned}
\shfdiv{D} & = \sum_i\left(\irrdivcoeff_i +  D_i \cdot D \right) D_i + \sum_\alpha \irrdivcoeff_\alpha D_\alpha + \sum_\mu \irrdivcoeff_\mu D_\mu ~\\
& = \sum_i \left(\sum_\alpha\irrdivcoeff_\alpha(D_i\cdot D_\alpha) +\sum_\mu\irrdivcoeff_\mu(D_i\cdot D_\mu)\right) D_i + \sum_\alpha \irrdivcoeff_\alpha D_\alpha + \sum_\mu \irrdivcoeff_\mu D_\mu  \,,
\end{aligned}
\eenn
where we have used again the relation $D_i \cdot D_{j} = -\delta_{ij}$.
The only negative contribution in $\shfdiv{D} \cdot D_\alpha$ is $\irrdivcoeff_\alpha D_\alpha^2=-\irrdivcoeff_\alpha$. Now, if $D_i \cdot D_\alpha = 0$ for all $i$, then $\shfdiv{D} \cdot D_\alpha = D \cdot D_\alpha\geq 0$ from the definition of $\tilde D$. So assume there is at least one $i$ such that $D_i \cdot D_\alpha > 0$. But then from the above expression for $\tilde D$ it follows that
\benn
\shfdiv{D} \cdot D_\alpha = \irrdivcoeff_\alpha(D_i \cdot D_\alpha)^2 - \irrdivcoeff_\alpha + \text{ non-negative terms } \geq 0 \,.
\eenn
Hence $\shfdiv{D}$ is in the nef cone. 
\end{proof}

\subsection{Topological formulae for the zeroth cohomology}
The isoparametric transform reduces the problem of computing $h^0(S,\mcL)$ for an arbitrary effective line bundle $\mcL=\cO(D)$ to the treatment of nef line bundles alone. 

In the situations when all higher cohomologies $h^{i>0}(S,\mc{O}(\underline{\tilde D})$ vanish, the zeroth cohomology dimension of $\cO(\underline{\tilde D})$ and hence $h^0(S,\cO(D))$ can be computed as a topological index,
\be\label{IndexFormula2}
h^0(S,\cO(D)) = \chi(S,\cO(\underline{\tilde D}))~.
\ee

For certain surfaces this gives a topological formula for the dimension of the zeroth cohomology throughout the entire effective cone. As stated in Corollary~\ref{crl:KV}, this is always the case for surfaces with nef and big anti-canonical divisor class. 

\medskip\noindent\textbf{Proof of Corollary~\ref{crl:KV}.} 
Let $D$ be an effective divisor and $\underline{\tilde D}$ its isoparametric transform limit. The Kawamata-Viehweg vanishing theorem applies to  $\underline{\tilde D}$, in the sense that all higher cohomologies $h^{i>0}(S,\mc{O}(\underline{\tilde D}))$ vanish, if $\underline{\tilde D}-K_S$ is nef and big. 

In particular, if $K_S$ is nef and big, then nefness of $\underline{\tilde D} - K_S$ is immediate, since $\underline{\tilde D}$ is always nef. 
Recalling that a nef divisor is big if and only if its top self-intersection is strictly positive, bigness of $\underline{\tilde D} - K_S$ follows from
\benn
(\shfdiv{D}-K_S)^2 = \shfdiv{D}^2 + 2\shfdiv{D}\cdot(-K_S) + (-K_S)^2 >0 \,,
\eenn
since $(-K_S)^2>0$, as $-K_S$ is nef and big.
\qed

\subsection{Back-engineering the negative curves.} So far we have discussed the isoparametric transform as a means of procuring topological formulae for cohomology. The other aspect of the isoparametric transform is that, given (at least some partial) knowledge of line bundle cohomology on a given surface, the negative curves on that surface can be worked out.

\section{Calculations on del Pezzo surfaces}\label{sec:Examples}

\begin{thm}
\label{thm:dpcohom}
On a del Pezzo surface, the zeroth cohomology of a line bundle $\mc{O}(D)$ associated to an effective divisor $D$ is given by
\benn
h^0\left(\mc{O}(D)\right) = \ind\bigg(D + \sum_{i}(D\cdot D_i)D_i \bigg) ~.
\eenn
where the index $i$ runs over the set of all irreducible negative divisors such that $D \cdot D_i < 0$. 
\end{thm}

\begin{proof}
Since the anti-canonical divisor of a del Pezzo surface is ample,  Corollary~\ref{crl:KV} implies that the zeroth cohomology of any effective line bundle can be computed using the topological formula~\eqref{IndexFormula}. Moreover, del Pezzo surfaces contain no curves with self-intersection less than $-1$, hence Theorem \ref{Thm:oneStep} implies that the isoparametric transform abuts after one iteration. 
\end{proof}

\vspace{12pt}
Note that the list of exceptional curves on a del Pezzo surface is well-known and short, so in practice it is easy to compute the isoparametric transform. Additionally, the list of generators of the nef cone is also well-known and relatively short, so since this cone is dual to the Mori cone, effectiveness of divisors can easily be checked using intersection properties.

\section{Hirzebruch Surfaces}\label{Hirzebruch}

As a particular example of compact toric surfaces, consider the Hirzebruch surface $\mbb{F}_{n>0}$ containing a single negative irreducible divisor, $D_{-n}^2=-n$.  

\begin{thm}
On the Hirzebruch surface $\mbb{F}_n$, the zeroth cohomology of the line bundle $\mc{O}(D)$ associated to an effective divisor $D$ is given by
\be
h^0\left(\mc{O}(D)\right) = \ind\bigg(D - \theta(-D \cdot D_{-n})\,\ceil{\frac{D\cdot D_{-n}}{-n}}D_{-n} \bigg) \,,
\ee
where $\theta$ is the Heaviside function.
\end{thm}
\begin{proof}
If $D$ is nef, by Demazure's vanishing theorem $h^0\left(\mc{O}(D)\right) = \chi(D)$. If $D$ is effective but not nef, $D\cdot D_{-n}<0$ and the isoparametric transform abuts after one iteration, since
\benn
\left(D - \ceil{\frac{D\cdot D_{-n}}{D_{-n}^2}}D_{-n} \right) \cdot D_{-n} = D \cdot D_{-n} - \ceil{\frac{D\cdot D_{-n}}{D_{-n}^2}}(D_{-n}^2) \geq 0 \,.
\eenn
The result is nef, and the Demazure vanishing theorem gives the result.
\end{proof}

\bibliographystyle{halphanum}
\bibliography{references}

\begin{thebibliography}{GSY2}

\bibitem[BCDL1]{BCDL1} C.~Brodie, A.~Constantin, R.~Deen and A.~Lukas, \textit{Index Formulae for Line Bundle Cohomology on Complex Surfaces}, Fortsch.~Phys.~68 (2020) 1900086, \href{https://arxiv.org/abs/1906.08769}{arXiv:1906.08769}. 

\bibitem[BCDL2]{BCDL2} C.~Brodie, A.~Constantin, R.~Deen and A.~Lukas, \textit{Machine Learning Line Bundle Cohomology}. Fortsch.~Phys.~68 (2020) 1900087, \href{https://arxiv.org/abs/1906.08730}{arXiv:1906.08730}.

\bibitem[BCL]{BCL} E.~Buchbinder, A.~Constantin, and A.~Lukas, \textit{The Moduli Space of Heterotic Line Bundle Models: a Case Study for the Tetra-Quadric}, JHEP 1403 (2014) 025,  \href{https://arxiv.org/abs/1311.1941}{arXiv:1311.1941}. 

\bibitem[CL]{CL} A.~Constantin, and A.~Lukas, \textit{Formulae for Line Bundle Cohomology on Calabi-Yau Threfolds}. \href{https://arxiv.org/abs/1808.09992}{arXiv:1808.09992}.

\bibitem[CLS]{CLS}  D. Cox, J. Little and H. Schenck, \textit{Toric Varieties}, Graduate Studies in Mathematics. American Mathematical Society (2011).

\bibitem[Co]{Co} A.~Constantin, \textit{Heterotic String Models on Smooth Calabi-Yau Threefolds}, DPhil thesis, University of Oxford (2013). \href{https://arxiv.org/abs/1808.09993}{arXiv:1808.09993}.

\bibitem[KS]{KS} D.~Klaewer and L.~Schlechter, \textit{Machine Learning Line Bundle Cohomologies of
Hypersurfaces in Toric Varieties}, Phys.Lett. B789 (2019) 438-443. \href{http://arxiv.org/abs/1809.02547}{arXiv:1809.02547}.

\bibitem[LS]{LS} M.~Larfors and R.~Schneider, \textit{Line bundle cohomologies on CICYs with Picard number two}, Fortsch.~Phys.~67 (2019) 1900083 \href{http://arxiv.org/abs/1906.00392}{arXiv:1906.00392}.

\bibitem[La]{La} R. Lazarsfeld, \textit{Positivity in algebraic geometry}, Springer, Berlin (2004).

\end{thebibliography}

\begin{multicols}{2}
\bigskip \noindent {\tt{andrei.constantin@physics.ox.ac.uk}} \\
\bigskip \noindent {\tt{rehan.deen@balliol.ox.ac.uk}} \medskip\\[-17pt]
\bigskip \noindent {\tt{lukas@physics.ox.ac.uk}} \medskip

\noindent Rudolf Peierls Centre for Theoretical Physics \\
\noindent University of Oxford \\
Oxford OX1 3PU \\
UK
\\\\\\\\
\bigskip \noindent {\tt{callum.brodie@ipht.fr}} \\

\noindent Institut de Physique Th\'{e}orique, \\
Universit\'{e} Paris Saclay, CEA, CNRS \\ 
\noindent Orme des Merisiers, 91191 Gif-sur-Yvette CEDEX\\
\noindent France \\

\end{multicols}

\end{document}